\documentclass[11pt]{article}
\usepackage{amsmath}
\usepackage{amssymb}
\usepackage{amsthm}
\usepackage{amscd}
\usepackage{amsfonts}
\usepackage{graphicx}%
\newtheorem{theorem}{Theorem}[section]
\newtheorem{lemma}[theorem]{Lemma}

\newtheorem{Example}{Example}[section]
\newtheorem{definition}{Definition}[section]

\newcommand{\block}[1]{
	\underbrace{\begin{matrix} n & \cdots & n \end{matrix}}_{#1} }
\begin{document}
\title{\bf Extension of Paley Construction for Hadamard Matrix}
\author {Shipra Kumari\thanks {E-mail: shipra.kumari@cuj.ac.in } ~and~ Hrishikesh Mahato\thanks{Corresponding author: E-mail: hrishikesh.mahato@cuj.ac.in}\\
	\small Department of Mathematics,
	\small Central University of Jharkhand,
	\small Ranchi-835205, India}
\date{}
\maketitle \setlength{\parskip}{.11 in}
\setlength{\baselineskip}{15pt}
\maketitle
\begin{abstract}
	We have extended the Paley constructions for Hadamard matrices and obtained some series of Hadamard matrices. Especially Paley construction-II is applicable for odd prime power $q \equiv 1(mod \ 4)$ however our method is applicable for any odd prime power. Some of which are non-isomorphic to the matrices obtained by standard Paley constructions, Sylvester construction, Williamson construction or combination of either of these constructions with Sylvester construction. In fact we have used the conference matrix of order prime power $q$ in some different manner which is applicable for any odd prime power $q$ or product of twin primes.\\

\end{abstract}

\hrule

\section{Introduction:}
A Hadamard Matrix  is a square matrix $ H_{n} $ of order $ n $ with entries from $ \{ -1 ,1\} $ which satisfies the condition $ H_{n}H_{n}^{T} = nI_{n}$, where $  I_{n}$ is the Identity matrix of order $ n$ and $H_{n}^T$ is the transpose of $H_{n}$. The order of Hadamard Matrix  is necessarily  $  1 $ , $  2 $ or is multiple of $  4 $. The converse of this statement seems to be true and is known as Hadamard Conjecture  \cite{hord}.\\
Historically, Hadamard matrices were first introduced by \textbf{James Joseph Sylvester} in $1867$ \cite{sylvester}.
Sylvester presented a construction by considering that if $H$ is a Hadamard matrix of order $ n$ then  
\begin{equation*}
	\left[
	\begin{array}{cc}
		H& H\\
		H&-H
	\end{array}
	\right]
\end{equation*} is also a Hadamard matrix of order $2n$.
Starting with $H=1$, we get a Hadamard matrix of order $2^n$ for all non-negative integer $n$
\[
H_{1}=
\begin{bmatrix}
1 
\end{bmatrix}, \ \ \ \ 
H_{2}=
\begin{bmatrix}
1&1 \\
1&-1
\end{bmatrix}, \ \ \ \ 
H_{2^n}=
\begin{bmatrix}
H_{2^n-1}&H_{2^n-1}\\
H_{2^n-1}&-H_{2^n-1}
\end{bmatrix} 
\]
In this manner Sylvester constructed Hadamard matrix of order $2^n$ for every positive integer $n$.\\
Another popular construction of Hadamard matrix is Paley construction which was described in $1933$ by the english mathematician \textbf{Raymond Paley} \cite{Pal}. The Paley construction uses quadratic residues of a finite field $GF(q),$ where $q$ is an odd prime power. Two methods of Paley constructions has been distinguished by the types of $q$ with $q \equiv 1(mod \ 4)$ and $q\equiv 3(mod \ 4)$. The properties of conference matrix are used in both the constructions. In Paley construction type-I for each $q \equiv 3 \ (mod \ 4)$ there exist a Hadamard matrix of order $q+1$ and in that of type-II for each $q\equiv 1 \ (mod \ 4)$ there exist a Hadamard matrix of order $2(q+1)$.\\ 
\hspace{5cm} In this article we use the conference matrices $Q$ of order $q=p^r$ ( $p$ is an odd prime and $r \in \mathbb{Z}^+$) in different manner and construct Hadamard matrices of order $n(q+1)$ for $q\equiv 3(mod \ 4)$ and that of order $2^{k+1}(q+1)$ for $q\equiv 1(mod \ 4)$, where $n$ is the order of a known Hadamard matrix and $k$ is any non-negative integer. It has been observed that some of Hadamard matrices of order $16, \ 40, \ 56, \cdots $ cannot be constructed by Paley construction-I and construction-II directly. Our aim is to extend the Paley construction-I and II so that we can  construct Hadamard matrices of above mentioned orders directly. 
Also we will discuss about existence of Hadamard matrices using twin primes and unifiy in same construction.
\section {Preliminaries}
\begin{definition} \cite{mar}\textbf{Extended Quadratic character}\\
	\ \ \  Let $ q=p^r$, where $p$ is an odd prime,  $r$ is any positive integer and  $ GF(q) = \{0= \alpha_{0}  ,  \alpha_{1}  ,\alpha_{2}  ,  . . . , \alpha_{q-1} \} $. Then the \textbf{Extended quadratic character} is a map $\chi $ defined on $GF(q)$  as 
	\begin{equation*}
		\chi (\alpha_{i})=
		\begin{cases}
			\ \  \ 1 \ ;\ \ \ \ \ \  if \ \alpha_{i} \ is \ quadratic \ residue \ in \ GF(q);\\
			\ \  \ 0 \ ; \ \ \ \ \ \ if \  \alpha_{i} =0\ ; \ \ \ \ \\\
			-1 \ ; \ \ \ \ \ \  if \ \alpha_{i} \ is \ quadratic \ nonresidue \ in \  GF(q).
		\end{cases}\
	\end{equation*}
	
\end{definition}
\begin{definition}\label{def2.2}\cite{mar}\textbf{Quadratic Residue}\\
 An element $a$ of $GF(q)$ is said to be \textbf{ quadratic residue} if it is a perfect square in $GF(q)$ otherwise $a$ is a 
	\textbf{quadratic non-residue}.
\end{definition}
\begin{definition}\label{def2.3} \cite{hord}\textbf{Kronecker Product} \\
	Let $ A= [a_{ij} ] $ and   $B= [b_{rs}]$ be two square  matrices of order $m$ and $n$ respectively. Then the Kronecker Product (Direct Product ) of $A$ and $B$ is a square matrix of order $mn$ and is given  by\\
	\begin{equation*}
		A\otimes B=\left[ \begin{array}{ccccc}
			a_{11}B &  a_{12}B & \textbf{. . . } & a_{1m}B \\
			a_{21}B &  a_{22}B & \textbf{. . . } & a_{2m}B \\
			\textbf{.} &  & \ \ \ \ \textbf{.} & \textbf{.} \\
			\textbf{.} &  & \hspace{1.5cm} \textbf{.} & \textbf{.} \\
			\textbf{.} & & \hspace{2.2cm} \textbf{.}& \textbf{.} \\
			a_{m1}B &  a_{m2}B &\textbf{ . . .} & a_{mm}B \\
		\end{array}
		\right]\\
	\end{equation*}
\end{definition}
\begin{definition} \cite{sign} \textbf{Sign changes of a sequence } \\ 
	The number of sign changes of a finite sequence $\{a_{k}\}_{k=1}^{n}, a_{k}=\pm 1$ is the number obtained by counting the number of times $+1$ is followed by $-1$ and that of times $-1$ is followed by $+1$.
	\end{definition}
	\begin{definition}\cite{sign}\textbf{Sign spectrum of a Hadamard matrix}\\
		The row (column) sign spectrum of an $n \times n $ Hadamard matrix $H$ is the sequence of numbers of sign changes appear in its rows (columns).
		\end{definition}
		
		\begin{definition}\label{def2.3}\textbf{Conference Matrix}\\ A Conference Matrix is a square matrix of order $q$ with diagonal entries $0$ and off-diagonal entries $\pm 1$ which satisfies the condition $QQ^T=qI_{q} -J_{q}.$ 
			
		\end{definition}

\begin{lemma}\cite{mar} Let $F$ be a field then $\sum_{b}\chi(b)\chi(b+c)=-1$ if $c \neq 0$,  where $b,c \in F $
\end{lemma}

\begin{lemma}\cite{mar}\label{lemma 2.2}
	The matrix Q = $ [q_{ij}]$ of order $q \times q$ where $q_{ij}=\chi(\beta_{j}-\beta_{i})$, $i, j =0,1,2,...,q-1$
	has the following properties\\
	(a) Q is symmetric if $ q\equiv 1\ (mod \ 4)$ 
	and  skew symmetric if $ q \equiv 3\ (mod \ 4)$\\
	(b) $QQ^{T}$ = $qI_{q}- J$\\  
\end{lemma}	
\section{Results:}
Consider $H_{n}=[h_{ij}]$ is a Hadamard matrix of order $n$. Let us define a rectangular matrix $E=[e_{ij}]$ of order $nq\times n$ whose elements are defined as follows
\begin{equation}{\label{eq 1}}
	e_{ij}=
	\begin{cases}
		h_{1j} & \hspace{5cm}if \ 1 \ \leq i \ \leq q;\\
		h_{2j} & \hspace{5cm} if  \ \ q< i \ \leq 2q\\
		
		. & \hspace{6cm} .\\
		. & \hspace{6cm}.\\
		. & \hspace{6cm}.\\
		h_{nj} & \hspace{5cm}  if \ (n-1)q < i \ \leq nq
		
	\end{cases}
\end{equation}
\begin{center}
	where  $j$= $1, 2, 3, \ . \ . \ . \ n$\\	
\end{center}

Infact, matrix $E=[e_{ij}]$ is formed by repeating each row of Hadamard matrix $H_{n}$ $q$ times successively. i.e. the first $q$ rows of $E$ are repetitions of $1^{st}$ row of $H_{n}$, $(q+1)^{th}$ row to $2q^{th}$ rows of $E$ are repetitions of $2^{nd}$ row of $H_{n}$ and so on.\\
Similarly the matrix $E'$ of order $n \times nq $ is  formed by repeating each column of Hadamard matrix $H_{n}$ $q$ times successively.
 \label{e}\\

Since rows and columns of Hadamard matrix is mutually orthogonal. So 
\begin{equation}{\label{eq:E'E'^T}}
E'E'^{T}=
\left[
\begin{array}{ccccccc}
nq && \ 0 & &  \cdots& & 0\\
0 && nq & & \cdots & & 0 \\
\\
\vdots & & &  \ddots &&&\vdots\\
\\
0& & &  \cdots  & && nq
\end{array}
\right]_{n\times n}=nqI_{n}
\end{equation}\\
and,
\begin{equation} {\label{eq:H_{n}E^T}}
H_{n}E^T =	\left[
\begin{array}{cccccccccc}

\smash[b]{\block{q}} && 0 &  & \cdots  && 0  \\
\\
0& \smash[b]{\block{q}}& &&&& 0 \\
&&&&\ddots &\\
\vdots &&&&&& \vdots\\
0&&&&&&\block{q}\\

\end{array}
\right]_{n \times nq}
\end{equation}




Similarly
\begin{equation} \label{eq: EH_{n}^T}
EH_{n}^T=(H_{n}E^T)^T
\end{equation}
and
\begin{equation}{\label{eq:EE^T}}
EE^T=
\left[
\begin{array}{ccccccccc}
nJ_{q} & 0 &\cdots && 0\\
0 & nJ_{q} & \cdots && 0\\
\vdots && \ddots && \vdots \\
0 && \cdots &&nJ_{q}
\end{array}
\right]_{nq \times nq} = nJ_{q}\otimes I_{n}
\end{equation}
where $J_{q}$ is a square matrix of order $q$ with all entries $1$.

\begin{theorem}\label{theorem:first}Let $q =p^r$, where $p$ is an odd prime and $r$ is any positive integer, such that $q\equiv 3 \ ( mod \ 4)$, then there exist a Hadamard matrix of order $nq+n$, where $n$ is the order of a known Hadamard Matrix.
 
\end{theorem}
\begin{proof}
Consider a square matrix \begin{equation*}
A=[Q+I]\otimes H_{n}		
\end{equation*}
 of order $nq$, where $Q$ is a conference matrix of order $q$.\\ 
 Then we see that 
\begin{equation*}
K= \left[
\begin{array}{cc}
-H_{n} & E'\\
E & A	
\end{array}
\right]	
	\end{equation*}
	forms a Hadamard matrix of order $nq+n$.   

	 We have 
\begin{align*}
	KK^{T} &=\left[
	\begin{array}{cc}
		-H_{n} & E'  \\
		E & A
	\end{array}
	\right] \left[
	\begin{array}{cc}
		-H_{n} & E'\\
		E & A\\
	\end{array}
	\right]^{T}\\
	&=\left[
	\begin{array}{cc}
		-H_{n} & E^{T}  \\
		E & A
	\end{array}
	\right] \left[
	\begin{array}{cc}
		-H_{n}^{T} & E^{T}  \\
		E'^{T} & A^{T}
	\end{array}
	\right]
\end{align*}
\begin{align}\label{eq:KK^T1}
 \hspace{2.7cm} &	= \left [
	\begin{array}{cc}
		H_{n}H_{n}^{T}+E'E'^{T} &  \ \ -H_{n}E^T+E'A^{T}\\
		-EH_{n}^{T}+AE'{T} & \ \ EE^{T}+AA^{T}
	\end{array}
	\right] 
\end{align}
 
Now \begin{align*}
AA^T &=([Q+I]\otimes H_{n}) \ ([Q+I]\otimes H_{n})^T\\
&=([Q+I]\otimes H_{n}) \ ([Q^T+I]\otimes H_{n}^T)\\
&=([QQ^T+Q+Q^T+I]) \otimes H_{n}H_{n}^T\\
& =([QQ^T+Q+Q^T+I])\otimes nI_{n}\\
&=([QQ^T+I])\otimes nI_{n} \hspace{3cm} \mbox{As, $Q^T=-Q$ for $q\equiv 3(mod \ 4)$} \\
&=([qI_{q}-J_{q}+I_{q}]) \otimes nI_{n}	\\
&=((q+1)I_{q}-J_{q})\otimes nI_{n}	\\
 & = n(q+1)I_{q}\otimes I_{n}-nJ_{q}\otimes I_{n}
\end{align*}
As \begin{equation*}
EE^T=nJ_{q}\otimes I_{n}
\end{equation*}
So,
\begin{equation*}
 AA^T+EE^T= n(q+1)I_{q}\otimes I_{n}-nJ_{q}\otimes I_{n}+nJ_{q}\otimes I_{n}	
\end{equation*}
\begin{equation}\label{eq:AA^T+EE^T1}
\Rightarrow AA^T+EE^T= n(q+1)I_{q}\otimes I_{n}
\end{equation}

 As the number of quadratic residues and that of quadratic non-residues in $GF(q)$ are same which is $\frac{q-1}{2}$ by the definition of $Q$ the number of $+1's$ and $-1's$ in each row and column of matrix $Q+I$ are $\frac{q-1}{2}+1$  and $\frac{q-1}{2}$ respectively. Therefore
\begin{equation}\label{eq:E'A^T1}
	E'A^{T}=
\def\onegroup{\underbrace{n \hskip\arraycolsep\cdots \hskip\arraycolsep n}_{q}}
	\left[
	\begin{array}{cccccc}
	\onegroup & &&  &&  \\
	&& \onegroup && & \\
\\ 	&& &\ddots && \\
\\ 	&&  &&& \onegroup
 \end{array}
  \right]_{n \times nq}
\end{equation}

 		
 		
 		

 Using equation \eqref{eq:H_{n}E^T} and equation \eqref{eq:E'A^T1}
 \begin{equation}\label{eq:E'A^T-H_{n}E^T1}
 E'A^T-H_{n}E^T= \textbf{0}
 \end{equation}
 Taking transpose both sides\\
 \begin{equation}\label{eq:AE'^T-EH_{n}^T1}
 AE'^T-EH^T_{n}= \textbf{0}
  \end{equation}
From the equation \eqref{eq:E'E'^T} and Hadamard matrix $H_{n}$  we have\\
\begin{equation}\label{eq:H_{n}H_{n}^T+E'E'^T1}
H_{n}H_{n}^T+E'E'^T=nI_{n}+nqI_{n}=(nq+n)I_{n}
\end{equation}
Combining equation \eqref{eq:AA^T+EE^T1}, \eqref{eq:E'A^T-H_{n}E^T1}, \eqref{eq:AE'^T-EH_{n}^T1} , \eqref{eq:H_{n}H_{n}^T+E'E'^T1} equation \eqref{eq:KK^T1} becomes,
\begin{equation*}
	KK^T=
	\left[
	\begin{array}{ccccc}
		nq+n&0&\cdots& \ \ 0\\
		\ \ 0&nq+n&\cdots& \ \ 0\\
		\vdots&& \ddots& \ \ \vdots\\
		0& \ \ \cdots&&nq+n
	
	\end{array}
	\right] = (nq+n)I_{nq+n}
\end{equation*}
Therefore $K$ is a Hadamard matrix of order $nq+n$. \end{proof}
\textbf{Remark:} In particular this method coincides with Paley construction-I for $n=1$ 
\begin{theorem}\label{thm2.1}Let $q =p^r$, where $p$ is an odd prime and $r$ is any positive integer, such that $q\equiv 1 \ ( mod \ 4)$, then there exist a Hadamard matrix of order $2^{k+1}(q+1)$, where $k$ is any non-negative integer.
	\end{theorem}
\begin{proof}
	Consider a square matrix
	\begin{equation}
	A=\left[Q_{q}\otimes\pm (RH_2R)+I_{q}\otimes H_2\right]\otimes H_{2^{k}}
	\end{equation}
	where \begin{equation*}
	H_{2^k}= \underbrace{H_{2}\otimes H_{2}\otimes \cdots \otimes H_{2}}
	\end{equation*}  and $Q$ is a conference matrix of order $q$. Then we see that
	\begin{equation*}
	K=	\left[
	\begin{array}{cc}
	-H_{n} & E'\\
	E & A
	\end{array}
	\right]
	\end{equation*} is a Hadamard matrix of order $2^{k+1}(q+1)$, where $H_{n}=H_{2^k}\otimes H_{2}$.\\ 
	 
  We have 
\begin{align*}
KK^T &=	\left[
\begin{array}{cc}
	-H_{n} & E'\\
	E & A	
\end{array}
\right]
\left[\begin{array}{cc} -H_{n} & E'\\
	E & A
	\end{array}
\right]^T\\
&= \left[ \begin{array}{cc}
	-H_{n} & E'\\
	E & A \end{array}
\right] 
\left[ \begin{array}{cc}
	-H_{n}^T & E^T\\
	E'^{T} & A^T
	\end{array}
\right]
\end{align*}
\begin{align} \label{eq:KK^T2}
\hspace{2.3cm}&= \left[ \begin{array}{cc}
	H_{n}H_{n}^T+E'E'^{T} & -H_{n}E^T+E'A^T\\
	-EH_{n}^T+AE'^{T} & EE^T+AA^T
	\end{array}
\right]
\end{align} 
Calculations of $H_{n}H_{n}^T+E'E'^{T}, -H_{n}E^T+E'A^T, -EH_{n}^T+AE'^{T}$ and $EE^T$ are similar to the theorem \eqref{theorem:first}. To show the result it is enough to calculate $AA^T$. \\ 
Here 
\begin{align*}
AA^T& = \left( \left[Q\otimes(\pm RH_{2}R)+ I\otimes H_{2}\right]\otimes H_{2^k}\right)\left(  \left[Q\otimes(\pm RH_{2}R)+ I\otimes H_{2}\right]\otimes H_{2^k} \right)^T\\
&=\left(Q \otimes (\pm RH_{2}R)\otimes H_{2^k}+I_{q}\otimes H_{2}\otimes H_{2^k} \right)\left(Q^T \otimes(\pm R^TH_{2}^TR^T)\otimes H_{2^k}^T+I_{q}\otimes H_{2}^T\otimes H_{2^{k}}^T\right)\\
&=\left(Q \otimes (\pm RH_{2}R)\otimes H_{2^k}\right)\left(Q^T\otimes (\pm R^TH_{2}^TR^T)\otimes H_{2^k}^T\right)+ \left(Q\otimes (\pm RH_{2}R)\otimes H_{2^k}\right)\left(I_{q}\otimes H_{2}^T\otimes H_{2^k}^T\right)\\
& \ \ \ \ \ \ \ \  +\left(I_{q}\otimes H_{2}\otimes H_{2^k}\right)\left(Q^T\otimes(\pm R^TH_{2}^TR^T)\otimes H_{2^k}^T\right)+ \left(I_{q}\otimes H_{2}\otimes H_{2^k}\right)\left(I_{q}\otimes H_{2}^T\otimes H_{2^k}^T\right)\\
&=\left(Q\otimes (\pm RH_{2}R)\right)\left(Q^T\otimes(\pm R^TH_{2}^TR^T)\right)\otimes H_{2^k}H_{2^k}^T+ \left(Q\otimes (\pm RH_{2}R)\right)\left(I_{q}\otimes H_{2}^T\right)
\otimes H_{2^k}H_{2^k}^T\\
&\ \ \ \ \ \ \ \  +\left(I_{q}\otimes H_{2}\right)\left(Q^T\otimes (\pm R^TH_{2}^TR^T)\right)\otimes H_{2^k}H_{2^k}^T +\left(I_{q}\otimes H_{2}\right)\left(I_{q}\otimes H_{2}^T\right)\otimes H_{2^k}H_{2^k}^T\\
&=\left(QQ^T\otimes RH_{2}RR^TH_{2}R^T\right)\otimes H_{2^k}H_{2^k}^T+\left(Q\otimes \pm RH_{2}RH_{2}^T\right)\otimes H_{2^k}H_{2^k}^T\\
& \ \ \ \ \ \ \ \ +\left(Q^T\otimes \pm H_{2}R^TH_{2}^TR^T\right)\otimes H_{2^k}H_{2^k}^T+
 I_{q}\otimes H_{2}H_{2}^T\otimes H_{2^k}H_{2^k}^T\\
 &=\left(QQ^T\otimes  RH_{2}(RH_{2})^T\right)\otimes 2^kI_{2^k}+\left(Q\otimes \pm RH_{2}RH_{2}^T\right)\otimes 2^kI_{2^k}+\left(Q^T\otimes \pm (RH_{2}RH_{2}^T)^T\right)\otimes 2^kI_{2^k}\\
 &\ \ \ \ \ \ \ \ +I_{q}\otimes 2I_{2}\otimes 2^kI_{2^k}\\
 &=\left((qI_{q}-J_{q})\otimes  2I_{2}\right)\otimes 2^kI_{2^k}+\left(Q\otimes \pm RH_{2}RH_{2}^T\right)\otimes 2^kI_{2^k} +\left(Q\otimes \pm (RH_{2}RH_{2}^T)^T \right)\otimes 2^kI_{2^k} +2^{k+1}I_{2^{k+1}q} \\
 & \hspace{4cm}\mbox{As $Q^T=Q$ for $q \equiv 1(mod \ 4)$}\\
 &=\left(2qI_{q}\otimes I_{2}-2J_{q}\otimes I_{2}\right)\otimes 2^kI_{2^k}+\left(Q\otimes \pm RH_{2}RH_{2}^T\right)\otimes 2^kI_{2^k}- \left(Q\otimes \pm RH_{2}RH_{2}^T\right)\otimes 2^kI_{2^k}+2^{k+1}I_{2^{k+1}q}\\
 & \hspace{4cm}\mbox{It can be easily verified that $(RH_{2}RH_{2}^T)^T=-RH_{2}RH_{2}^T$}\\
 &= 2^{k+1}qI_{2^{k+1}q}-2^{k+1}J_{q}\otimes I_{2^{k+1}}+ 2^{k+1}I_{2^{k+1}q}
\end{align*}
 As, here $H_{n}=H_{2^{k+1}}$ so using equation \eqref{eq:EE^T} \begin{equation*}
 EE^T=2^{k+1}J_{q}\otimes I_{2^{k+1}}
 \end{equation*}
 Thus \begin{align}
 AA^T+EE^T &= 2^{k+1}qI_{2^{k+1}q}-2^{k+1}J_{q}\otimes I_{2^{k+1}}+ 2^{k+1}I_{2^{k+1}q}+ 2^{k+1}J_{q}\otimes I_{2^{k+1}} \nonumber\\
 &= 2^{k+1}qI_{2^{k+1}q}+2^{k+1}I_{2^{k+1}q} \nonumber\\
 &=2^{k+1}(q+1)I_{2^{k+1}q}
 \end{align}

 Combining all the results equation \eqref{eq:KK^T2} becomes  $KK^T=2^{k+1}(q+1)I_{2^{k+1}(q+1)}$. Thus $K$ is a Hadamard matrix of order $2^{k+1}(q+1)$. \end{proof}

\begin{theorem} Let $p$ and $q$ are twin primes. Then there exist a Hadamard matrix of order $npq+n$, where $n$ is the order of a known Hadamard matrix.	
\end{theorem}
\begin{proof}  Consider \begin{equation*}B= R \otimes H_{n},\end{equation*} where $R$ is a $(-1,1)$ conference matrix of order $pq$ i.e. $RR^T=(pq+1)I_{pq}-J_{pq}$. $R$ may be constructed using difference set defined on product of twin primes \cite{mar}. Then we see that
\begin{equation}\label{eq:K3}
K= \left[ \begin{array}{cc}
H_{n} & E'\\
E & B
\end{array}
\right]
\end{equation}
forms a Hadamard matrix of order $npq+n$, where $H_{n}$ is a known Hadamard matrix of order $n$.\\
We have 
\begin{align}\label{eq:KK^T3}
 KK^T&= \left[ \begin{array}{cc}
		H_{n}H_{n}^T+E'E'^{T} & \ \ \  H_{n}E^T+E'B^T\\
		EH_{n}^T+BE'^{T} & \ \ \ EE^T+BB^T
	\end{array}
	\right]
\end{align} 
Here $E=[e_{ij}]$ is a matrix of order $npq \times n$ defined in equation $(1)$. Therefore \begin{equation*} E'E'^{T}=npqI_{n} \ \ and \ \ EE^T= nJ_{pq} \otimes I_{n}\end{equation*}

In the context of construction of $R$ it can be observed that there are  $\frac{(pq-1)}{2}$ elements in difference set modulo $pq$ \cite{mar}. So total number of $1's$ and $-1's$in each row and  column of $R$ are $\frac{pq-1}{2}$ and $\frac{pq+1}{2}$ respectively. 

\begin{equation}\label{eq:E'B^T3}
E'B^{T}=
\def\onegroup{\underbrace{-n \hskip\arraycolsep\cdots \hskip\arraycolsep -n}_{pq}}
\left[
\begin{array}{cccccc}
\onegroup & &&  &&  \\
&& \onegroup && & \\
\\ 	&& &\ddots && \\
\\ 	&&  &&& \onegroup \end{array} \right]_{n \times npq}
\end{equation}
Now
\begin{align*}
	BB^T &= (R \otimes H_{n}) \ (R \otimes H_{n})^T\\
	& = (R \otimes H_{n}) \ (R^T \otimes H_{n}^T)\\
	& = RR^T \otimes H_{n}H_{n}^T\\
	& = ((pq+1)I_{pq}-J_{pq}) \otimes nI_{n}\\
			& = n(pq+1)I_{npq}-(nJ_{pq}\otimes I_{n})	
				\end{align*}
and	\begin{equation*}
EE^T=nJ_{pq} \otimes I_{n}	
\end{equation*}
So \begin{align*}
BB^T+EE^T& = n(pq+1)I_{npq}-(nJ_{pq}\otimes I_{n})+ (nJ_{pq}\otimes I_{n})	
\end{align*}
\begin{align}\label{eq:BB^T+EE^T}
&= n(pq+1)I_{npq}
\end{align}
Using equation \eqref{eq:H_{n}E^T} and equation \eqref{eq:E'B^T3} \begin{equation}\label{eq:H_{n}E^T+E'B^T}
H_{n}E^T+E'B^T=\textbf{0}
\end{equation}
Taking transpose on both sides
 \begin{equation}\label{eq:EH_{n}^T+BE'^T}
 EH_{n}^T+BE'^T=\textbf{0}
 \end{equation}
 And
 \begin{equation}\label{eq:E'E'^T+H_{n}H_{n}^T}
 E'E'^{T}+H_{n}H_{n}^T= (npq+n)I_{n}\\
  \end{equation} 
  Combining equation \eqref{eq:BB^T+EE^T}, \eqref{eq:H_{n}E^T+E'B^T}, \eqref{eq:EH_{n}^T+BE'^T} and \eqref{eq:E'E'^T+H_{n}H_{n}^T} the  equation \eqref{eq:KK^T3} becomes
 $$KK^T= (npq+n)I_{npq+n}$$ 
 i.e. $K$ is a Hadamard matrix of order $npq+n$. 
 \end{proof}

Now, We will compare the isomorphism of the Hadamard matrices obtained by these methods and well known existing methods. For this we introduce the following definition of strictly normalized Hadamard matix.

 A \textbf{strictly normalized Hadamard matrix} of a given normalized Hadamard matrix $H$ is obtained by implementing suitable rows and/or columns permutations on $H$ so that all possible $1's$ to be shifted towards north west corner successively starting from first row and first column without affecting the previous rows or columns.\\
\begin{Example} Let $H$ be a normalized Hadamard matrix of order $4$ constructed by Paley construction-I
and strictly normalized Hadamard matrix of $H$ is denoted by $H_{s}$
\begin{equation*}
H= \left[ \begin{array}{cccc}
	1 &  \ 1 & \ 1 &\ 1\\
	\ 1 & -1 & -1 & 1\\
	\ 1 & \ 1 & - 1 & - 1\\
	\ 1 & - 1 & \ 1 & - 1
	\end{array}
\right]	, H_{s} =\left[ \begin{array}{cccc}
\ 1 & \ 1 & \ 1 & \ 1\\
\ 1 & \ 1 & -1 & -1\\
\ 1 & -1 & \ 1 & -1\\
\ 1 &  -1 & -1 & \ 1
\end{array}
\right]
\end{equation*}
\end{Example}
 Strictly normalized Hadamard matrix $H_{s}$ obtained from Hadamard matrix $H$ is equivalent to $H$. Also
 it is clear that two Hadamard matrices are isomorphic if and only if the set of sign spectrum of their strictly normalized Hadamard matrices are same. \\

\section{Observation}
Using the concept of strictly normalized Hadamard matrix we demonstrate some examples of non-isomorphic Hadamard matrices as follows

\begin{Example}	{Hadamard matrix of order $24$}

 In our method, described in this article, Hadamard matrix of order $24$ can be obtained in two ways taking $q=5$ and $q=11$. Suppose  Hadamard matrix of order $24$ obtained by using prime $q=11$ is $H^{1}$ and the  sign spectrum of the strictly normalized Hadamard matrix of $H^{1}$ is\\ 
$S(H_{1})= \{\{0, \ 1, \ 3, \ 2, \  7, \ 6, \  15, \ 14, \ 17, \ 18, \ 12, \ 13, \ 12, \ 13, \ 13, \ 14, \ 16, \ 15, \ 14, \ 13, \ 14, \ 15, \ 15, \ 14  \},\\
 \{0, \ 1, \ 3, \ 7, \ 8, \ 6, \ 6, \ 7, \ 8, \ 8, \ 7, \ 5, \ 23,  \ 20, \ 21, \ 18, \ 14, \ 19, \ 16, \ 15,  \ 14, \ 15, \ 17,  18  \} \}$.\\
 $H^{2}$ is the Hadamard matrix of order $24$ obtained by the method using prime $q=5$ and  the  sign spectrum of the strictly normalized Hadamard matrix of  $H^{2}$ is\\
 $S(H^{2})=\{ \{ 0,\ 1, \ 3, \ 7, \ 12, \ 13, \ 16, \ 12, \ 11, \ 10, \ 11, \ 10, \ 16, \ 10, \ 14, \ 13, \ 14, \ 13, \ 16, \ 17, \ 17, \ 17,  14, \ 9   \},\\
 \{0, \ 1, \ 3, \ 7, \ 11, \ 14, \ 9, \ 16, \ 14, \ 14, \ 13, \ 14, \ 18, \ 12, \ 14, \ 14, \ 15, \ 11, \ 12, \ 12, \ 13, \ 13, \ 11, \ 15 \} \}$\\
 $H^3$ is a Hadamard matrix of order $24$ obtained by Paley construction-I. The  sign spectrum of the strictly normalized Hadamard matrix of $H^{3}$ is \\
 $S(H^{3})= \{\{0, \ 1, \ 3, \ 7, \ 12, \ 13, \ 16, \ 12, \ 12, \ 13, \ 15, \ 10, \ 12, \ 14, \ 15, \ 14, \ 16, \ 13, \ 11, \ 14, \ 13, \ 15, \ 14, \ 11 \},\\
  \{ 0, \ 1, \ 3, \ 7, \ 12, \ 13, \ 13, \ 14, \ 10, \ 9, \ 14, \ 10, \ 14, \ 14, \ 17, \ 14, \ 15, \ 16, \ 14, \ 18, \ 13, \ 11, \ 11, \ 13 \}\} $\\
   $H^4$ is the Hadamard matrix of order $24$ obtained by Kronecker product of $H_{2}$ and $H_{12}$, where $H_{12}$ is constructed by Paley-I. Then sign spectrum of the strictly normalized $H^{4}$ is\\
$ S(H^{4})= \{\{0, \ 1, \ 3, \ 7, \ 13, \ 18, \ 14, \ 13, \ 12, \ 12, \ 12, \ 11, \ 14, \ 11, \ 12, \ 14, \ 14, \ 15, \ 14, \ 13, \ 15, \ 11, \ 17, \ 10\},\\
\{0, \ 1, \ 3, \ 7, \ 4, \ 5, \ 15, \ 18, \ 12, \ 12, \ 18, \ 11, \ 15, \ 12, \ 14, \ 18, \ 16, \ 15, \ 16, \ 15, \ 13, \ 13, \ 11, \ 12\} \}$.\\
$H^{5}$ is the Hadamard matrix of order $24$  obtained by Kronecker product of $ H_{2}$ and $H_{12}$, where $H_{12}$ is constructed by Paley-II, The  sign spectrum of strictly normalized Hadamard matrix of  $H^{5}$ is \\
$S(H^{5})=\{ \{ 0, \ 1, \ 3, \ 7, \ 6, \ 4, \ 15, \ 16, \ 14, \ 8, \ 19, \ 9, \ 14, \ 15, \ 16, \ 17, \ 16, \ 15, \ 16, \ 15, \ 13, \ 11, \ 12, \ 14\},\\
\{ 0, \ 1, \ 3, \ 7, \ 6, \ 4, \ 15, \ 16, \ 14, \ 8, \ 19, \ 9, \ 14, \ 16, \ 17, \ 17, \ 15, \ 14, \ 15, \ 14, \ 13, \ 12, \ 13, \ 14\} \}$\\

We observe that the sets of sign spectrum of $S(H^{1}), S(H^{2}), S(H^{3}), S(H^{4}), S(H^{5})$ are different.
So that we can say that the Hadamard matrix of order $24$ obtained by our method is non-isomorphic to the previously constructed Hadamard matrix of order $24$.\\
\end{Example}
\begin{Example}

 Hadamard matrix of order $28$\\
It can be constructed by Paley-I, Paley-II and our method.
$H^{1}$ is the Hadamard matrix of order $28$ obtained by paley construction-I. The sign spectrum of Strictly normalized Hadamard matrix $H^{1}$ is \\
$S(H^{1})=\{ \{0, \ 1, \ 3, \ 7, \ 13, \ 14, \ 14, \ 14, \ 20, \ 19, \ 14, \ 17, \ 18, \ 15, \ 15, \ 14, \ 16, \ 14, \ 16, \ 15, \ 15, \ 18, \\ 14, \ 19, \ 15, \ 15, \ 15, \ 8  \},\\
 \{0, \ 1, \ 3, \ 7, \ 11, \ 14, \ 14, \ 18, \ 16, \ 16, \ 11, \ 17, \ 17, \ 16, \ 14, \ 16, \ 14, \ 20, \ 19, \ 17, \ 19, \ 17, \\ 13, \ 10, \ 16, \ 13, \ 17, \ 12  \} \}$\\
$H^{2}$ is a Hadamard matrix of order $24$ obtained by Paley construction-II. The sign spectrum of strictly normalized of $H^{2}$ is\\
$S(H^{2})= \{\{0, \ 1, \ 3, \ 7, \ 15, \ 22, \ 16, \ 20, \ 16, \ 13, \ 12, \ 15, \ 13, \ 14, \ 20, \ 14, \ 13, \ 15, \ 16, \ 14, \ 15, \ 16, \\ 13, \ 15, \ 14, \ 13, \ 14, \ 19  \},\\
\{ 0, \ 1, \ 3, \ 7, \ 15, \ 22, \ 16, \ 16, \ 13, \ 20, \ 11, \ 16, \ 13, \ 14, \ 16, \ 13, \ 20, \ 14, \ 17, \ 14, \ 15, \ 16, \\ 13, \ 16, \ 11, \ 20, \ 11, \ 15\} \}$\\
$H^{3}$ is a Hadamard matrix of order $28$ of order $28$ constructed by Williamson construction. The sign spectrum of strictly normalized of $H^{3}$ is\\
$S(H^{3}) = \{ \{  0, \ 1, \ 3, \ 7, \ 9, \ 16, \ 13, \ 16, \ 16, \ 10, \ 14, \ 14, \ 17, \ 15, \ 14, \ 18, \ 14, \ 19, \ 18, \ 19, \ 16, \ 19 ,\\ 14,\ 13, \ 14, \ 19, \ 17, \ 13\},\\ \{ 0, \ 1, \ 3, \ 7, \ 9, \ 20, \ 16, \ 16, \ 14, \ 14, \ 13, \ 14, \ 15, \ 13, \ 16, \ 17, \ 19, \ 18, \ 16, \ 17, \ 20, \ 12, \\ 12, \ 17, \ 19, \ 12, \ 13, \ 15 \}$\\
$H^{4}$ is a Hadamard matrix of order $28$ constructed by our method taking $q=13$. The sign spectrum of strictly normalized of $H^{4}$ is\\
$S(H^{4})=\{ 0, \ 1, \ 3, \ 7, \ 13, \ 18, \ 18, \ 14, \ 20, \ 15, \ 16, \ 16, \ 11, \ 17, \ 16, \ 16, \ 14, \ 12, \ 19, \ 15, \ 19, \ 17, \\ 17, \ 14, \ 17, \ 10, \ 17, \ 6\},\\
\{ 0, \ 1, \ 3, \ 7, \ 15, \ 17, \ 15, \ 16, \ 20, \ 16, \ 14, \ 18, \ 14, \ 13, \ 12, \ 12, \ 16, \ 16, \ 18, \ 16, \ 15, \ 17, \\ 15, \ 19, \ 15, \ 13, \ 13, \ 12\} \}$\\
We see that set of sign spectrum  $S(H^{1}),S( H^{2}), S(H^{3}), S(H^{4})$ are all different
So we can say that Hadamard matrix of order $28$ constructed by our method is non-isomorphic to the previously constructed Hadamard matrix by Paley construction-I, Paley construction-II and Williamson construction.
\end{Example}

\textbf{Remark:} It has been observed that some of Hadamard matrices obtained by our method are non-equivalent to Hadamard matrices obtained by Paley constructions, combination of Kronecker product with Paley construction and Williamson construction.

  \section{Conclusion}
There are many article published in direction of construction of Hadamard matrices. Paley Construction is one of the popular and classic method of construction. Paley has given two methods on the basis of types of primes and popularly known as Paley construction-I and Paley Construction-II. In this article, Hadamard matrix of known order is used to generalize Paley construction-I and obtained a series of Hadamard matrices. Paley construction-II is also extended using Sylvester Hadamard matrix. Also existence of Hadamard matrices using twin primes has been discussed and unified in same construction.\\
  It has been observed that some Hadamard matrices obtained by this method is non-isomorphic to previously constructed Hadamard matrices.

\end{document}